\documentclass{amsart}
\usepackage{graphicx}
\newtheorem{thm}{Theorem}[section]
\newtheorem{cor}[thm]{Corollary}

\theoremstyle{definition}

\theoremstyle{remark}
\newtheorem{rem}[thm]{Remark}
\numberwithin{equation}{section}
\begin{document}

\title[ harmonic number identities]{Combinatorial identities involving harmonic numbers}%
\author{necdet batir}%
\address{department of mathematics\\
faculty of sciences and arts\\
nev{\c{s}}ehir hac{\i} bekta{\c{s}} veli university, nev{\c{s}}ehir, turkey}
\email{nbatir@hotmail.com}%
%\thanks{}%

\date{October 30, 2017}%
%\dedicatory{}%
%\commby{}%
% ----------------------------------------------------------------
\begin{abstract}
In this work we prove a new combinatorial identity and applying it we establish many finite  harmonic sum identities. Among many others, we prove that
\begin{equation*}
  \sum_{k=1}^{n}\frac{(-1)^{k-1}}{k}\binom{n}{k}H_{n-k}=H_n^2+\sum_{k=1}^{n}\frac{(-1)^{k}}{k^2\binom{n}{k}},
\end{equation*}
and
\begin{equation*}
\sum_{k=1}^{n}\frac{(-1)^{k-1}}{k^2}\binom{n}{k}H_{n-k}=\frac{H_n[H_n^2+H_n^{(2)}]}{2}-\sum_{k=0}^{n-1}\frac{(-1)^k[H_n-H_k]}{(k+1)(n-k)\binom{n}{k}}.
\end{equation*}
Almost all of our results are new, while a few of them recapture know results.
\end{abstract}
\maketitle
% ----------------------------------------------------------------
\section{Introduction}
Let $s\in\mathbb{C}$. Then, the generalized harmonic numbers $H_n^{(s)}$ of order $s$ are defined by
\begin{equation}\label{e:1} %=================================== Equation (1.1)
H_n^{(s)}=\sum\limits_{k=1}^{n}\frac{1}{k^s}, \quad H_0^{(s)}=0 \quad \mbox{and} \quad H_n^{(1)}=H_n,
\end{equation}
see \cite{1,13}. These numbers have various applications in number theory, combinatorics, analysis, computer science and differential equations.  Recently, they have found applications in evaluating Feynman diagrams contributions of perturbed quantum field theory, see \cite{24,25}. Here and in the following we let  $\mathbb{N}_0=\mathbb{N}\cup\{0\}$, \, $\mathbb{Z}^-=\{-1,-2,-3,\cdots\}$ \, and  \,$\mathbb{Z}_0^-=\mathbb{Z}^-\cup\{0\}$.   The polygamma functions $\psi^{(n)}(s)$  ($s\in\mathbb{C}\backslash \mathbb{Z}^-)$ are defined by
\begin{equation}\label{e:2} %=================================== Equation (1.2)
\psi^{(n)}(s)=\frac{d^{n+1}}{ds^{n+1}}\log\Gamma(s)=\frac{d^{n}}{ds^{n}}\psi(s), n=0,1,2,\cdots,
\end{equation}
where $\Gamma(s)$ is the classical Euler's gamma function, and $\psi^{(0)}(s)=\psi(s)$ is the digamma function. Let us recall some basic properties of these functions which will be used frequently in this work.  A well-known relationship between the polygamma $\psi^{(n)}(s)$ and the generalized harmonic numbers $H_n^{(s)}$ is given by
\begin{equation}\label{e:3} %=================================== Equation (1.3)
\psi^{(m-1)}(n+1)-\psi^{(m-1)}(1)=(-1)^{m-1}(m-1)!H^{(m)}_n ,
\end{equation}
for $m\in\mathbb{N}$ and $n\in\mathbb{N}_0$; see \cite[pg. 22]{6}. The digamma function $\psi$ and harmonic numbers $H_n$ are related with
\begin{equation}\label{e:4} %=================================== Equation (1.4)
\psi(n+1)=-\gamma+H_n \quad(n\in\mathbb{N}),
\end{equation}
see \cite[pg. 31]{23}, where  $\gamma=0.57721\cdots$ is Euler-Mascheroni constant. The digamma function $\psi$ possesses the following properties:
\begin{equation}\label{e:5} %=================================== Equation (1.5)
\psi\left(s+\frac{1}{2}\right)=2\psi(2s)-\psi(s)-2\log2\quad (s\in\mathbb{C}\backslash\mathbb{Z}^-),
\end{equation}
and
\begin{equation}\label{e:6} %=================================== Equation (1.6)
\psi(s)-\psi(1-s)=-\pi\cot(\pi s),
\end{equation}
see \cite[pg. 25]{23}. The gamma function satisfies the reflection formula
\begin{equation}\label{e:7} %=================================== Equation (1.7)
\Gamma(s)\Gamma(1-s)=\frac{\pi}{\sin(\pi s)} \quad (s\in\mathbb{C}\backslash\mathbb{Z}^-)
\end{equation}
and the duplication formula
\begin{equation}\label{e:8} %=================================== Equation (1.8)
\Gamma\left(s+\frac{1}{2}\right)=\frac{\Gamma(2s)\Gamma(1/2)}{2^{2s-1}\Gamma(s)},
\end{equation}
see \cite[pgs. 346, 349]{16}. The binomial coefficients $\binom{s}{t}$ $(s,t\in\mathbb{C}\backslash\mathbb{Z}^-)$ are defined by
\begin{equation}\label{e:9} %=================================== Equation (1.9)
\binom{s}{t}=\frac{\Gamma(s+1)}{\Gamma(t+1)\Gamma(s-t+1)},
\end{equation}%
and they satisfy for for $n,k\in\mathbb{N}$ with $k\leq n$
\begin{equation}\label{e:10} %=================================== Equation (1.10)
\binom{n+1}{k}=\binom{n}{k}+\binom{n}{k-1} \quad \textrm{and}\quad\frac{n}{k}\binom{n-1}{k-1}=\binom{n}{k}.
\end{equation}
The beta function $B(s,t)$  is defined by
$$
B(s,t)=\int\limits_{0}^{1}u^{s-1}(1-u)^{t-1}du\quad (\Re(s)>0\,,\Re(t)>0).
$$
The gamma and beta functions are  related with
\begin{equation}\label{e:11} %=================================== Equation (1.11)
B(s,t)=\frac{\Gamma(s)\Gamma(t)}{\Gamma(s+t)},
\end{equation}
see \cite[p.251]{9}. In this paper we shall frequently use the following form of the binomial coefficients
$$
f_n(s):=\binom{s+n}{k}=\frac{\Gamma(s+n+1)}{k!\Gamma(s+n-k+1)}.
$$
Taking the logarithm of both sides of this equation, we have
$$
\log(f_n(s))=\log\Gamma(s+n+1)-\log\Gamma(s+n-k+1)-\log k!.
$$
Differentiation with respect to $s$ both sides  gives
\begin{equation}\label{e:12} %=================================== Equation (1.12)
f_n'(s)=\binom{s+n}{k}\{\psi(s+n+1)-\psi(s+n-k+1)\}.
\end{equation}
Let us also define
\begin{equation*}
g_n(s)=\binom{s+n}{n}.
\end{equation*}
Then differentiating yields
\begin{equation}\label{e:13} %=================================== Equation (1.13)
g_n'(s)=\binom{s+n}{n}[\psi(s+n+1)-\psi(s+1)],
\end{equation}
In the literature there exist many interesting identites for finite sums involving  the harmonic numbers and the binomial coefficients. For example, we have
\begin{equation*}
\sum\limits_{k=0}^{n}\binom{n}{k}^2H_k=\binom{2n}{n}[2H_n-H_{2n}] \quad (\textrm{see \cite{3}}),
\end{equation*}
\begin{equation*}
\sum_{k=0}^{n}(-1)^k\binom{n}{k}\{H_k-2H_{2k}\}=\frac{4^n}{n}\binom{2n}{n}^2 \quad (\textrm{see \cite{22}}),
\end{equation*}
and
\begin{equation*}
\sum\limits_{k=0}^{n}(-1)^k\binom{n}{k}H_{n+k}^2= \frac{1}{n\binom{2n}{n}}\left\{H_n-H_{2n}-\frac{2}{n}\right\},
\end{equation*}
see \cite[Corollary 23]{22}.

In the decades, combinatorial identities involving harmonic numbers have attracted the interest of many mathematicians and it has been discovered many interesting identities in different forms by these researchers by using different methods. In  \cite{4} and \cite{24} the authors used some identities of classical hypergeometric functions. In \cite{17} the authors computed the family of the following sums
\begin{equation*}
  \sum_{k=0}^n\binom{n}{k}^m\{1+m(n-2k)H_k\}\quad m=1,2,3,4,5
\end{equation*}
by using differential operator and Zeilberger's algorithm for definite hypergeometric sums.  Please see [2,3,4,5,12,14,15,17,18,21,22,24,25] and the references therein for more identities on this issue. The aim of this paper is to present further interesting combinatorial identities involving harmonic numbers. First, we establish a new combinatorial identity involving two parameters, and differentiating and integrating each side of this identity with respect to these parameters we obtain many harmonic number identities, some of which are new, and the others recover known identities.  Although, many other combinatorial identities can be derived by using these identities, for briefness we have selected here just some of them, and we intend to prepare a separate paper containing many other applications.

Now we are ready to present our main results.

\section{Main results}
\begin{thm}% ================================ Theorem 2.1
Let $n\in\mathbb{N}$, $s\in\mathbb{C}\backslash\mathbb{Z}^-$ and $x\in\mathbb{C}$. Then there holds the following identity:
\begin{equation}\label{e:14}%============================================================================================================================================ Equation (2.1)
\sum\limits_{k=0}^{n}\binom{s+n}{k}x^k=(1+x)^n\left[1+s\sum_{k=0}^{n-1}\frac{\binom{s+k}{k}}{k+1}\left(\frac{x}{x+1}\right)^{k+1}\right].
\end{equation}
\end{thm}
\begin{proof} We prove by mathematical induction. Clearly, (\ref{e:14}) is valid for $n=1$. We assume that it is valid for $n$ and we shall show that it is also valid for $n+1$. We have
\begin{equation*}
\sum\limits_{k=0}^{n+1}\binom{s+n+1}{k}x^k=\sum\limits_{k=0}^{n}\binom{s+n+1}{k}x^k+\binom{s+n+1}{n+1}x^{n+1}.
\end{equation*}
Using the first relation in (\ref{e:10}), this becomes
\begin{equation*}
\sum\limits_{k=0}^{n+1}\binom{s+n+1}{k}x^k=\sum\limits_{k=0}^{n}\binom{s+n}{k}x^k+\sum\limits_{k=1}^{n}\binom{s+n}{k-1}x^k+\binom{s+n+1}{n+1}x^{n+1}.
\end{equation*}
Setting $k-1=k'$ in the second sum on the right hand side and then dropping the prime on $k'$, we get after a simple computation:
\begin{align*}
\sum\limits_{k=0}^{n+1}\binom{s+n+1}{k}x^k&=\sum_{k=0}^{n}\binom{s+n}{k}x^k+x\sum_{k=0}^{n}\binom{s+n}{k}x^k\\
&+\left[\binom{s+n+1}{n+1}-\binom{s+n}{n}\right]x^{n+1}.
\end{align*}
By (\ref{e:10}) we have $\binom{s+n+1}{n+1}-\binom{s+n}{n}=\binom{s+n}{1+n}$, thus, we get
\begin{equation}\label{e:15}%============================================================================================================================================ Equation (2.2)
\sum\limits_{k=0}^{n+1}\binom{s+n+1}{k}x^k=(1+x)\sum\limits_{k=0}^{n}\binom{s+n}{k}x^k+\binom{s+n}{1+n}x^{1+n}.
\end{equation}
Therefore, by induction hypothesis we deduce
\begin{align*}
&\sum\limits_{k=0}^{n+1}\binom{s+n+1}{k}x^k=(1+x)^{n+1}+s(1+x)^n\sum\limits_{k=0}^{n-1}\binom{s+k}{k}\frac{x^{k+1}}{(k+1)(1+x)^k}\\
&+\binom{s+n}{n+1}x^{n+1}=(1+x)^{n+1}+s(1+x)^n\sum\limits_{k=0}^{n}\binom{s+k}{k}\frac{x^{k+1}}{(k+1)(1+x)^k}\\
&+\bigg[\binom{s+n}{1+n}-\frac{s}{n+1}\binom{s+n}{n}\bigg]x^{n+1}.
\end{align*}
Since $\binom{s+n}{n+1}-\frac{s}{n+1}\binom{s+n}{n}=0$, this proves that (\ref{e:14}) is also valid for $n+1$. This completes the proof.
\end{proof}
\begin{thm} For $n\in\mathbb{N}$, $s\in\mathbb{C}\backslash\mathbb{Z}^-$ and $x\in\mathbb{C}$ we have   %================================== Theorem 2.2
\begin{align}\label{e:16}%============================================================================================================================================ Equation (2.3)
&\sum_{k=1}^{n}\binom{s+n}{k}\{\psi(s+n+1)-\psi(s+n-k+1)\}x^k\notag\\
&=(1+x)^n\sum_{k=0}^{n-1}\frac{\binom{s+k}{k}}{k+1}\{1+s[\psi(s+k+1)-\psi(s+1)]\}\left(\frac{x}{x+1}\right)^{k+1}.
\end{align}
\end{thm}
\begin{proof}
The proof follows from differentiating with respect to $s$ both sides of (\ref{e:14}), and using  (\ref{e:12}) and (\ref{e:13}).
\end{proof}
\begin{cor} For $n\in\mathbb{N}$ and $s\in\mathbb{C}\backslash\mathbb{Z}^-$ we have %================================== Corrollary 2.3
\begin{align}\label{e:17} %=========================================================================================================================================== Equation (2.4)
\sum_{k=0}^{n}&(-1)^k\binom{s+n}{k}[\psi(s+n+1)-\psi(s+n-k+1)]\notag\\
&=\frac{(-1)^n}{n}\binom{s+n-1}{n-1}[1+s(\psi(s+n)-\psi(s+1))].
\end{align}
\end{cor}
\begin{proof}
The proof immediately follows from (\ref{e:16}) by taking $x=-1$.
\end{proof}
\begin{thm} For $n\in\mathbb{N}$, $s\in\mathbb{C}\backslash\mathbb{Z}^-$ and $x\in\mathbb{C}$ we have %=============================================Theorem 2.4
\begin{align}\label{e:18} %============================================================================================================================================ Equation (2.5)
&\sum_{k=0}^{n}\binom{s+n}{k}\bigg\{(\psi(s+n+1)-\psi(s+n-k+1))^2\notag\\
&+\psi'(s+n+1)-\psi'(s+n-k+1)\bigg\}x^k\notag\\
&=(1+x)^n\sum_{k=0}^{n-1}\frac{1}{k+1}\binom{s+k}{k}\bigg\{2(\psi(s+k+1)-\psi(s+1))\notag\\
&+s\bigg[(\psi(s+k+1)-\psi(s+1))^2+\psi'(s+k+1)-\psi'(s+1)\bigg]\bigg\}\left(\frac{x}{x+1}\right)^{k+1}.
\end{align}
\end{thm}
\begin{proof}
The proof follows from differentiating with respect to $s$ both sides of (\ref{e:16}), by the help of (\ref{e:12}) and (\ref{e:13}).
\end{proof}
\begin{cor}For $n\in\mathbb{N}$ and $s\in\mathbb{C}\backslash\mathbb{Z}^-$ we have  %=====================================Corollarry  2.5
\begin{align}\label{e:19} %============================================================================================================================================ Equation (2.6)
&\sum_{k=0}^{n}(-1)^k\binom{s+n}{k}\bigg\{(\psi(s+n+1)-\psi(s+n-k+1))^2\notag\\
&+\psi'(s+n+1)-\psi'(s+n-k+1)\bigg\}\notag\\
&=\frac{(-1)^n}{n}\binom{s+n-1}{n-1}\bigg\{2[\psi(s+n)-\psi(s+1)]\notag\\
&+s\left[\left(\psi(s+n)-\psi(s+1)\right)^2+\psi'(s+n)-\psi'(s+1)\right]\bigg\}.
\end{align}
\end{cor}
\begin{proof}
If we write Eq. (\ref{e:18}) at $x=-1$, the proof is completed.
\end{proof}
\begin{thm} For $n\in\mathbb{N}$ and $s\in\mathbb{C}\backslash\mathbb{Z}^-$ we have  %========================================== Theorem 2.6
\begin{equation}\label{e:20} %============================================================================================================================================ Equation (2.7)
\sum_{k=1}^{n}\binom{s+n}{k}\frac{(-1)^{k-1}}{k}=H_n+ s\sum_{k=0}^{n-1}\frac{(-1)^k\binom{s+k}{k}}{(k+1)^2\binom{n}{k+1}}.
\end{equation}
\end{thm}
\begin{proof} Taking the first term of the sum on the left side  of (\ref{e:14}) to the right, dividing each side by $x$, and finally replacing $x$ by $-x$ in the resulting equation, we obtain
\begin{equation}\label{e:21} %=============================================================================================================================================================================== Equation (2.8)
\sum_{k=1}^{n}(-1)^k\binom{s+n}{k}x^{k-1}=-\frac{(1-x)^n-1}{(1-x)-1}+ s\sum_{k=0}^{n-1}\frac{(-1)^k\binom{s+k}{k}}{k+1}(1-x)^{n-k-1}x^k.
\end{equation}
Integrating each side of (\ref{e:21}) from $x=0$ to $x=1$, we find that
\begin{align}\label{e:22} %============================================================================================================================================================================== Equation (2.9)
&\sum_{k=1}^{n}\binom{s+n}{k}\frac{(-1)^{k}}{k}\notag\\
&=-\int_{0}^{1}\frac{(1-x)^n-1}{(1-x)-1}\thinspace dx+ s\sum_{k=0}^{n-1}\frac{(-1)^k\binom{s+k}{k}}{k+1}\int_0^1(1-x)^{n-k-1}x^k\thinspace dx,
\end{align}
which, in view of
\begin{equation*}
\int_{0}^{1}\frac{(1-x)^n-1}{(1-x)-1}\thinspace dx=H_n,
\end{equation*}
see \cite{24}, and
\begin{equation*}
\int_0^1(1-x)^{n-k-1}x^k\thinspace dx=\frac{(n-k-1)!k!}{n!},
\end{equation*}
we complete the proof.
\end{proof}
\begin{thm} For all $n\in\mathbb{N}$ and $s\in\mathbb{C}\backslash\mathbb{Z^-}$ we have %========================================== Theorem 2.7
\begin{align}\label{e:23}%=================================== Equation (2.10)
&\sum_{k=1}^{n}\frac{(-1)^{k-1}}{k}\binom{s+n}{k}[\psi(s+n+1)-\psi(s+n-k+1)]\notag\\
&=\sum_{k=0}^{n-1}\frac{(-1)^k\binom{s+k}{k}}{(k+1)^2\binom{n}{k+1}}+s\sum_{k=0}^{n-1}\frac{(-1)^k\binom{s+k}{k}}{(k+1)^2\binom{n}{k+1}}[\psi(s+k+1)-\psi(s+1)].
\end{align}
\end{thm}
\begin{proof}
The proof follows immediately from differentiating both sides of (\ref{e:20}) with respect to $s$, and using (\ref{e:12}) and (\ref{e:13}).
\end{proof}
\begin{thm} For all $n\in\mathbb{N}$ and $s\in\mathbb{C}\backslash\mathbb{Z^-}$ we have %========================================== Theorem 2.8
\begin{align}\label{e:24} %=================================== Equation (2.11)
&\sum_{k=1}^{n}\frac{(-1)^{k-1}}{k}\binom{s+n}{k}\bigg\{(\psi(s+n+1)-\psi(s+n-k+1))^2\notag\\
&+\psi'(s+n+1)-\psi'(s+n-k+1)\bigg\}\notag\\
&=2\sum_{k=0}^{n-1}\frac{(-1)^k\binom{s+k}{k}}{(k+1)^2\binom{n}{k+1}}[\psi(s+k+1)-\psi(s+1)]\notag\\
&+s\sum_{k=0}^{n-1}\frac{(-1)^k\binom{s+k}{k}}{(k+1)^2\binom{n}{k+1}}\bigg\{(\psi(s+k+1)-\psi(s+1))^2\notag\\
&+\psi'(s+k+1)-\psi'(s+1)\bigg\}.
\end{align}
\begin{proof}
If we differentiate both sides of (\ref{e:23}) with respect to $s$ the proof is completed.
\end{proof}
\end{thm}
\begin{thm} For all $n\in\mathbb{N}$ and $s\in\mathbb{C}\backslash\mathbb{Z^-}$ we have %========================================== Theorem 2.9
\begin{equation}\label{e:25} %=================================== Equation (2.12)
\sum\limits_{k=1}^{n}\binom{s+n}{k}\frac{(-1)^{k-1}}{k^2}=\frac{H_n^2+H_n^{(2)}}{2}+s\sum_{k=0}^{n-1}\frac{(-1)^k}{k+1}\binom{s+k}{k}\frac{H_n-H_k}{(n-k)\binom{n}{k}}.
\end{equation}
\end{thm}
\begin{proof}
Integrating both sides of (\ref{e:21}) from $x=0$ to $x=u$, we get
\begin{align*}
\sum\limits_{k=1}^{n}(-1)^k&\binom{s+n}{k}\frac{u^k}{k}=\int_{0}^{u}\frac{(1-x)^n-1}{x}dx\\
&+s\sum_{k=0}^{n-1}\frac{(-1)^k\binom{s+k}{k}}{k+1}\int_{0}^{u}(1-x)^{n-k-1}x^kdx.
\end{align*}
Integrating both sides of this equation over $[0,1]$, after dividing by $u$ each side, we obtain
\begin{align}\label{e:26} %=================================== Equation (2.13)
\sum\limits_{k=1}^{n}(-1)^k&\binom{s+n}{k}\frac{1}{k^2}=\int_{0}^{1}\frac{1}{u}\int_{0}^{u}\frac{(1-x)^n-1}{x}dxdu\notag\\
&+s\sum_{k=0}^{n-1}\frac{(-1)^k\binom{s+k}{k}}{k+1}\int_{0}^{1}\frac{1}{u}\int_{0}^{u}(1-x)^{n-k-1}x^kdxdu.
\end{align}
For the first integral on the right hand side of (\ref{e:26}), integration by parts gives
\begin{align}\label{e:27} %=================================== Equation (2.14)
\int_{0}^{1}\frac{1}{u}\int_{0}^{u}\frac{(1-x)^n-1}{x}dxdu&=\log u\int_{0}^{u}\frac{(1-x)^n-1}{x}dx\bigg|_{u=0}^{u=1}\notag\\
&-\int_{0}^{1}\frac{\log u}{u}[(1-u)^n-1]du.
\end{align}
The first term on the right hand side of (\ref{e:27}) is equal to zero, so that, applying integration by parts to the last integral, we get
\begin{align}\label{e:28} %=================================== Equation (2.15)
\int_{0}^{1}\frac{1}{u}\int_{0}^{u}\frac{(1-x)^n-1}{x}dxdu&=-\frac{1}{2}\log^2u[(1-u)^n-1]\bigg|_{u=0}^{u=1}\notag\\
&-\frac{n}{2}\int_{0}^{1}\log^2 u(1-u)^{n-1}du.\notag\\
\end{align}
The first term on the right hand side of (\ref{e:28}) is equal to zero, hence, we get by (\ref{e:11}) and (\ref{e:3})
\begin{align}\label{e:29} %=================================== Equation (2.16)
\int_{0}^{1}\frac{1}{u}\int_{0}^{u}\frac{(1-x)^n-1}{x}dxdu&=-\frac{n}{2}\int_{0}^{1}\log^2u (1-u)^{n-1}du\notag\\
&=-\frac{n}{2}\frac{d^2}{dt^2}\int_{0}^{1}u^t(1-u)^{n-1}du\bigg|_{t=0}\notag\\
&=-\frac{n}{2}\frac{d^2}{dt^2}\frac{\Gamma(t+1)\Gamma(n)}{\Gamma(t+n+1)}\bigg|_{t=0}=-\frac{H_n^2+H_n^{(2)}}{2}.
\end{align}
For the second integral in (\ref{e:26}) we have by integration by parts
\begin{align*}
&\int_{0}^{1}\frac{1}{u}\int_{0}^{u}(1-x)^{n-k-1}x^kdxdu=\log u\int_{0}^{u}(1-x)^{n-k-1}x^kdx\bigg|_{u=0}^{u=1}\notag\\
&-\int_{0}^{1}\log u(1-u)^{n-k-1}u^kdu.
\end{align*}
The first term on the right hand side of this equation is zero. So, we get by (\ref{e:11}) and (\ref{e:3})
\begin{align}\label{e:30} %=================================== Equation (2.17)
&\int_{0}^{1}\frac{1}{u}\int_{0}^{u}(1-x)^{n-k-1}x^kdxdu=-\int_{0}^{1}\log u(1-u)^{n-k-1}u^kdu\notag\\
&=-\int_{0}^{1}\frac{d}{dt}u^t(1-u)^{n-k-1}du\bigg|_{t=k}=-\frac{d}{dt}\int_{0}^{1}u^t(1-u)^{n-k-1}du\bigg|_{t=k}\notag\\
&=-\frac{d}{dt}\frac{\Gamma(t+1)\Gamma(n-k)}{\Gamma(n+t-k+1)}\bigg|_{t=k}=\frac{H_n-H_k}{(n-k)\binom{n}{k}}.
\end{align}
Inserting the values of the integrals given in (\ref{e:29}) and (\ref{e:30}) in (\ref{e:26}) and using (\ref{e:3}), we complete the proof of Theorem 2.9.

\end{proof}
\begin{thm} For all $n\in\mathbb{N}$ and $s\in\mathbb{C}\backslash\mathbb{Z^-}$ we have %========================================== Theorem 2.10
\begin{align}\label{e:31} %=================================== Equation (2.18)
&\sum\limits_{k=1}^{n}\binom{s+n}{k}\frac{(-1)^{k-1}}{k^2}[\psi(s+n+1)-\psi(s+n-k+1)]\notag\\
&=\sum_{k=0}^{n-1}\frac{(-1)^k}{k+1}\binom{s+k}{k}\frac{H_n-H_k}{(n-k)\binom{n}{k}}\notag\\
&+s\sum_{k=0}^{n-1}\frac{(-1)^k}{k+1}\binom{s+k}{k}\frac{H_n-H_k}{(n-k)\binom{n}{k}}[\psi(s+k+1)-\psi(s+1)].
\end{align}
\end{thm}
\begin{proof}
The proof follows from differentiating both sides of the equation (\ref{e:25}) with respect to $s$, and using (\ref{e:12}) and (\ref{e:13}).
\end{proof}
\begin{thm} For all $n\in\mathbb{N}$ and $s\in\mathbb{C}\backslash\mathbb{Z^-}$ we have %========================================== Theorem 2.11
\begin{align}\label{e:32} %=================================== Equation (2.19)
&\sum\limits_{k=1}^{n}\binom{s+n}{k}\frac{(-1)^{k-1}}{k^2}\bigg\{(\psi(s+n+1)-\psi(s+n-k+1))^2\notag\\
&+\psi'(s+n+1)-\psi'(s+n-k+1)\bigg\}\notag\\
&=2\sum_{k=0}^{n-1}\frac{(-1)^k}{k+1}\binom{s+k}{k}\frac{H_n-H_k}{(n-k)\binom{n}{k}}[\psi(s+k+1)-\psi(s+1)]\notag\\
&+s\sum_{k=0}^{n-1}\frac{(-1)^k}{k+1}\binom{s+k}{k}\frac{H_n-H_k}{(n-k)\binom{n}{k}}\bigg\{(\psi(s+k+1)-\psi(s+1))^2\notag\\
&+\psi'(s+k+1)-\psi'(s+1)\bigg\}.
\end{align}
\end{thm}
\begin{proof}
The proof follows from differentiating both sides of the equation (\ref{e:31}) with respect to $s$, and using (\ref{e:12}) and (\ref{e:13}).
\end{proof}
\section{Applications}
In this section we present many applications of our main results, which are derived by taking particular values for the parameters $s$ and $x$.\\
\\
\textbf{Identity 1}. For $n\in\mathbb{N}$ it holds that
\begin{equation}\label{e:33} %=================================== Equation (3.1)
\sum\limits_{k=0}^{n}(-1)^k\binom{s+n}{k}=(-1)^n\binom{s+n-1}{n}.
\end{equation}
\begin{proof}
Making use of (\ref{e:10}), the proof immediately follows from setting  $x=-1$ in  (\ref{e:14}).
\end{proof}
Performing the replacement $s\to s-n$ in (\ref{e:33}), and using (\ref{e:10}) we get:\\
\textbf{Identity 2.} For all $n\in\mathbb{N}$ and $s\in\mathbb{N}$ with $s\geq n$ or $s\in\mathbb{C}\backslash\mathbb{Z}$ we have
\begin{equation}\label{e:34} %=================================== Equation (3.2)
\sum_{k=0}^{n}(-1)^k\binom{s}{k}=(-1)^n\binom{s-1}{n}.
\end{equation}
\textbf{Identity 3.} Let $n\in\mathbb{N}$. Then we have
\begin{equation*}
\sum_{k=0}^{n}\frac{1}{2^{2k}}\binom{2k}{k}=\frac{2n+1}{2^{2n}}\binom{2n}{n}.
\end{equation*}
\begin{proof} Setting $s=-1/2$ in (\ref{e:34}), we get
\begin{equation}\label{e:35} %=================================== Equation (3.3)
\sum_{k=0}^{n}(-1)^k\binom{-1/2}{k}=(-1)^n\binom{-3/2}{n}.
\end{equation}
Clearly, by (\ref{e:9}), we have
\begin{equation}\label{e:36}%=================================== Equation (3.4)
\binom{-1/2}{k}=\frac{\Gamma(1/2)}{\Gamma\left(\frac{1}{2}-k\right)k!}.
\end{equation}
Using the reflection and duplication formulas given in (\ref{e:7}) and (\ref{e:8}) with $s=\frac{1}{2}-k$, we see that
\begin{equation}\label{e:37} %=================================== Equation (3.5)
\Gamma\left(\frac{1}{2}-k\right)\Gamma\left(\frac{1}{2}+k\right)=\frac{\pi}{\sin\left[\pi\left(\frac{1}{2}-k\right)\right]}=(-1)^k\pi.
\end{equation}
Using (\ref{e:8}) with $s=k$, $\Gamma(1/2)=\sqrt{\pi}$, and (\ref{e:37}), we obtain
\begin{equation*}
\Gamma\left(\frac{1}{2}-k\right)=\frac{(-1)^k\sqrt{\pi}2^{2k}k!}{(2k)!}.
\end{equation*}
Inserting this into (\ref{e:36}) we get
\begin{equation}\label{e:38} %=================================== Equation (3.6)
\binom{-1/2}{k}=\frac{(-1)^k}{2^{2k}}\binom{2k}{k}.
\end{equation}
A similar computation leads to
\begin{equation}\label{e:39}%=================================== Equation (3.7)
\binom{-3/2}{n}=(-1)^n\frac{2n+1}{2^{2n}}\binom{2n}{n}.
\end{equation}
Replacing (\ref{e:38}) and (\ref{e:39}) in (\ref{e:35}), the desired result is obtained.
\end{proof}
\textbf{Identity 4.} For $n\in\mathbb{N}$, we have
\begin{equation*}
 \sum_{k=1}^{n}\frac{1}{2^{2k}}\binom{2k}{k}(2H_{2k}-H_k)=\frac{2n+1}{2^{2n}}\binom{2n}{n}\left[2H_{2n}-H_n-\frac{4n}{2n+1}\right].
\end{equation*}
\begin{proof} If we set $s=-1/2$ in (\ref{e:34}), after differentiating with respect to $s$ each side,  we get
\begin{align}\label{e:40} %=================================== Equation (3.8)
\sum_{k=1}^{n}&(-1)^k\binom{-1/2}{k}\left[\psi(1/2)-\psi(1/2-k)\right]\notag \\
&=(-1)^n\binom{-3/2}{n}\left[\psi(-1/2)-\psi(-n-1/2)\right]
\end{align}
Setting $s=\frac{1}{2}-k$ in (\ref{e:6}), we obtain
\begin{equation*}
\psi(1/2-k)-\psi(1/2+k)=-\cot(\pi/2-\pi k)=0.
\end{equation*}
Using (\ref{e:5}) this yields
\begin{equation*}
\psi(1/2-k)=\psi(1/2+k)=2\psi(2k)-\psi(k)-2\log 2,
\end{equation*}
so that,
\begin{equation*}
\psi(1/2)-\psi(1/2-k)=\psi(1/2)+2\log2-2\psi(2k)+\psi(k).
\end{equation*}
Since $\psi(1/2)+\gamma+2\log2=0$ (see \cite[pg. 32]{3}), we conclude by (\ref{e:4})
\begin{equation}\label{e:41} %=================================== Equation (3.9)
\psi(1/2)-\psi(1/2-k)=-2H_{2k}+H_k.
\end{equation}
A similar calculation yields
\begin{equation}\label{e:42} %=================================== Equation (3.10)
\psi(-1/2)-\psi(-1/2-n)=\frac{4n}{2n+1}-2H_{2n}+H_n.
\end{equation}
Hence, we immediately obtain, by employing (\ref{e:41}) and (\ref{e:42}) in (\ref{e:40})
\begin{align*}
\sum_{k=1}^{n}&(-1)^k\binom{-1/2}{k}[-2H_{2k}+H_k]\\
&=(-1)^n\binom{-3/2}{n}\left[\frac{4n}{2n+1}-2H_{2n}+H_n\right].
\end{align*}
Taking into account Equations (\ref{e:38}) and (\ref{e:39}), we complete the proof.
\end{proof}
\textbf{Identity 5.} For all $n\in\mathbb{N}$, the following identity holds.\\
\begin{equation}\label{e:43}%=================================== Equation (3.11)
\sum_{k=1}^{n}\frac{(-1)^{k-1}}{k}\binom{n}{k}=H_n.
\end{equation}
\begin{proof}
The proof follows from (\ref{e:3}) by setting  $s=0$ in (\ref{e:20}).
\end{proof}
\begin{rem}
Identity (\ref{e:43}) is well-known and originally due to Euler (see \cite{10}, \cite{5} and \cite{2}).
\end{rem}
\begin{rem}
Identity  (\ref{e:20})  provides a generalization of (\ref{e:43}).
\end{rem}
\textbf{Identity 6.} For $n\in\mathbb{N}$ it holds that
\begin{equation}\label{e:44} %=================================== Equation (3.12)
  \sum_{k=1}^{n}\frac{(-1)^{k-1}}{k}\binom{n}{k}H_{n-k}=H_n^2+\sum_{k=1}^{n}\frac{(-1)^{k}}{k^2\binom{n}{k}}.
\end{equation}
\begin{proof}
The proof follows from (\ref{e:23}) by setting $s=0$, and using (\ref{e:43}) and (\ref{e:3}).
\end{proof}
\begin{rem}
Identity (\ref{e:44}) can be compared with the following identity (see \cite{19}):
\begin{equation*}
\sum_{k=1}^{n}\frac{(-1)^{k-1}}{k}\binom{n}{k}H_{n+k}=H_n^2+\sum_{k=1}^{n}\frac{1}{k^2\binom{n+k}{k}}.
\end{equation*}
\end{rem}
\textbf{Identity 7.} For all $n\in\mathbb{N}$ and $x\in\mathbb{C}$, we have
\begin{equation}\label{e:45} %=================================== Equation (3.13)
\sum_{k=0}^{n}\binom{n}{k}H_{n-k}x^k=(1+x)^n\left[H_n-\sum_{k=1}^{n}\frac{1}{k}\left(\frac{x}{1+x}\right)^k\right].
\end{equation}
\begin{proof} Taking  $s=0$ in (\ref{e:16}), we get
\begin{equation}\label{e:46} %=================================== Equation (3.14)
\sum_{k=0}^{n}\binom{n}{k}[\psi(n+1)-\psi(n-k+1)]x^k=(1+x)^n\sum_{k=1}^{n}\frac{1}{k}\left(\frac{x}{1+x}\right)^k.
\end{equation}
Employing (\ref{e:4}), (\ref{e:46}) implies that
\begin{equation}\label{e:47} %=================================== Equation (3.15)
\sum_{k=0}^{n}\binom{n}{k}[H_n-H_{n-k}]x^k=(1+x)^n\sum_{k=1}^{n}\frac{1}{k}\left(\frac{x}{1+x}\right)^k,
\end{equation}
which is equivalent with (\ref{e:45}), since $\sum_{k=0}^{n}\binom{n}{k}H_nx^k=(1+x)^nH_n$.
\end{proof}
If set $x=1$ in (\ref{e:45}), we get the following known result \cite{6,12,21}:\\
\textbf{Identity 8.} For $n\in\mathbb{N}$ it holds true
$$
\sum_{k=0}^{n}\binom{n}{k}H_k=2^n\left[H_n-\sum_{k=1}^{n}\frac{1}{k2^k}\right].
$$
\textbf{Identity 9.} For $n\in\mathbb{N}$ and $x\in\mathbb{C}$, we have
\begin{align}\label{e:48} %=================================== Equation (3.16)
&\sum_{k=1}^{n}\binom{n}{k}\left[H_k^2+H_k^{(2)}\right]x^k\notag\\
&=(1+x)^n\left[H_n^2+H_n^{(2)}+2\sum_{k=1}^{n}\frac{H_{k-1}-H_n}{k(1+x)^k}\right].
\end{align}
\begin{proof} By setting $s=0$ in (\ref{e:18}), we can readily deduce by the help of (\ref{e:3}) and (\ref{e:4}) that
\begin{align}\label{e:49} %=================================== Equation (3.17)
&\sum_{k=0}^{n}\binom{n}{k}\left[\left(H_n-H_{n-k}\right)^2-H_n^{(2)}+H_{n-k}^{(2)}\right]x^k\notag\\
&=2(1+x)^n\sum_{k=1}^{n}\frac{H_{k-1}}{k}\left(\frac{x}{1+x}\right)^k.
\end{align}
Expanding the quadratic term on the left side of (\ref{e:49})  and using (\ref{e:45}), we get after simplification
\begin{align}\label{e:50} %=================================== Equation (3.18)
&\sum_{k=0}^{n}\binom{n}{k}\left[H_{n-k}^2+H_{n-k}^{(2)}\right]x^k\notag\\
&-\left[H_n^2+H_n^{(2)}\right](1+x)^n+2H_n(1+x)^n\sum_{k=1}^{n}\frac{1}{k}\left(\frac{x}{1+x}\right)^k\notag\\
&=2(1+x)^n\sum_{k=1}^{n}\frac{H_{k-1}}{k}\left(\frac{x}{1+x}\right)^k.
\end{align}
Letting $n-k=k'$ on the first term in the left side of (\ref{e:50}) and then dropping the prime on $k'$, dividing both sides by $x^n$, then replacing
 $x$ by $1/x$ in both sides, and finally simplifying the resulting equation, we get the desired result (\ref{e:50}).
\end{proof}
If we set $x=-1$ in (\ref{e:48}), we get the following known result, see \cite{20}.\\
\textbf{Identity 10.} For $n\in\mathbb{N}$, we have
\begin{equation*}
\sum_{k=1}^{n}(-1)^k\binom{n}{k}\left\{H_k^2+H_k^{(2)}\right\}=-\frac{2}{n^2}.
\end{equation*}
The following identity is known and computer program \textit{Mathematica} recognize it. \\
\textbf{Identity 11.} For $n\in\mathbb{N}$, we have
\begin{equation}\label{e:51} %=================================== Equation (3.19)
\sum_{k=1}^{n}\frac{(-1)^k}{k\binom{n}{k}}=\frac{(-1)^n-1}{n+1}.
\end{equation}
\begin{proof} Setting $s=1$ in (\ref{e:20}), we get
\begin{equation}\label{e:52} %=================================== Equation (3.20)
\sum_{k=1}^{n}\frac{(-1)^{k-1}}{k}\binom{n+1}{k}=H_n+\sum_{k=1}^{n}\frac{(-1)^{k-1}}{k\binom{n}{k}}.
\end{equation}
Since
\begin{equation*}
\sum_{k=1}^{n}\frac{(-1)^{k-1}}{k}\binom{n+1}{k}=\sum_{k=1}^{n+1}\frac{(-1)^{k-1}}{k}\binom{n+1}{k}-\frac{(-1)^n}{n+1},
\end{equation*}
if we use (\ref{e:43}) we get
\begin{equation*}
\sum_{k=1}^{n}\frac{(-1)^{k-1}}{k}\binom{n+1}{k}=H_{n+1}-\frac{(-1)^n}{n+1}.
\end{equation*}
This completes the proof  of (\ref{e:51}) by the help of (\ref{e:52}).
\end{proof}
\textbf{Identity 12.} Let $n\in\mathbb{N}$. Then we have
\begin{align*}
&\sum_{k=1}^{n}\frac{(-1)^{k-1}}{k}\binom{n}{k}\left[H_{n-k}^2+H_{n-k}^{(2)}\right]\notag\\
&=H_n^3+H_nH_n^{(2)}+2\sum_{k=1}^{n}\frac{(-1)^k[H_n-H_{k-1}]}{k^2\binom{n}{k}}.
\end{align*}
\begin{proof} Putting $s=0$ in (\ref{e:24}), we obtain
\begin{equation*}
\sum_{k=1}^{n}\frac{(-1)^{k-1}}{k}\binom{n}{k}\left[\left(H_n-H_{n-k}\right)^2-H_n^{(2)}+H_{n-k}^{(2)}\right]\notag\\=2\sum_{k=1}^{n-1}\frac{(-1)^kH_k}{(k+1)^2\binom{n}{k+1}}.
\end{equation*}
Expanding the quadratic term here, this becomes
\begin{align*}
&\left[H_n^2-H_n^{(2)}\right]\sum_{k=1}^{n}\frac{(-1)^{k-1}}{k}\binom{n}{k}-2H_n\sum_{k=1}^{n}\frac{(-1)^{k-1}}{k}\binom{n}{k}H_{n-k}\notag\\
&+\sum_{k=1}^{n}\frac{(-1)^{k-1}}{k}\binom{n}{k}\left\{H_{n-k}^2+H_{n-k}^{(2)}\right\}=2\sum_{k=1}^{n-1}\frac{(-1)^kH_k}{(k+1)^2\binom{n}{k+1}}.
\end{align*}
Using  (\ref{e:43}) and (\ref{e:44}) here and simplifying the resulting equation we complete  the proof.
\end{proof}
\textbf{Identity 13.} Let $m,n\in\mathbb{N}$. Then we have
\begin{equation}\label{e:53} %=================================== Equation (3.21)
\sum_{k=0}^{n}(-1)^k\binom{mn}{k}H_{mn-k}=\frac{(-1)^{n}}{m}\binom{mn}{n}\left[(m-1)H_{(m-1)n}-\frac{1}{mn}\right].
\end{equation}
\begin{proof}
If we write  Eq. (\ref{e:34}) at  $s=mn$, we get
\begin{equation}\label{e:54} %=================================== Equation (3.22)
\sum_{k=0}^{n}(-1)^k\binom{mn}{k}=(-1)^n\binom{mn-1}{n}.
\end{equation}
If we differentiate with respect to $s$ both sides of (\ref{e:34}) and set $s=mn$ $(m\in\mathbb{N})$, we get, in view of (\ref{e:3})
\begin{align}\label{e:55} %=================================== Equation (3.23)
&\sum_{k=0}^{n}(-1)^k\binom{mn}{k}\{H_{mn}-H_{mn-k}\}\notag\\
&=(-1)^n\binom{mn-1}{n}\{H_{mn-1}-H_{mn-n-1}\}.
\end{align}
But if we use (\ref{e:54}) we can write
\begin{align*}
&\sum_{k=0}^{n}(-1)^k\binom{mn}{k}\{H_{mn}-H_{mn-k}\}\notag\\
&=H_{mn}\sum\limits_{k=0}^{n}(-1)^k\binom{mn}{k}-\sum\limits_{k=0}^{n}(-1)^k\binom{mn}{k}H_{mn-k}\notag\\
&=(-1)^nH_{mn}\binom{mn-1}{n}-\sum\limits_{k=0}^{n}(-1)^k\binom{mn}{k}H_{mn-k}.
\end{align*}
Using this identity in (\ref{e:55}) and simplifying the result we obtain
\begin{align*}
&\sum_{k=0}^{n}(-1)^k\binom{mn}{k}H_{mn-k}=(-1)^n\binom{mn-1}{n}\left[H_{mn-n}-\frac{1}{nm(m-1)}\right].
\end{align*}
Since $\binom{mn-1}{n}=\frac{m-1}{m}\binom{mn}{n}$, this completes the proof.
\end{proof}
For the special cases $m=2,3$, we get from (\ref{e:53}):\\
\textbf{Identity 14.} For $n\in\mathbb{N}$
\begin{align*}
 \sum_{k=0}^{n}(-1)^k\binom{2n}{k}H_{2n-k}&=\frac{(-1)^n}{2}\binom{2n}{n}\left(H_n-\frac{1}{2n}\right) \,\textrm{and}\\
 \sum_{k=0}^{n}(-1)^k\binom{3n}{k}H_{3n-k}&=\frac{(-1)^n}{3}\binom{3n}{n}\left(2H_n-\frac{1}{3n}\right).
\end{align*}
\textbf{Identity 15.} For $n\in\mathbb{N}$ we have
\begin{equation*}
\sum_{k=1}^{n}\frac{(-1)^{k}H_k}{k\binom{n}{k}}=\frac{(-1)^{n}H_{n+1}}{n+1}+\sum_{k=1}^{n+1}\frac{(-1)^{k}}{k^2\binom{n+1}{k}}.
\end{equation*}
\begin{proof} Setting $s=1$ in (\ref{e:23}), we obtain, in view of (\ref{e:3}) and $\psi(2)=1+\psi(1)$,
\begin{align}\label{e:56} %=================================== Equation (3.24)
&\sum_{k=1}^{n}\frac{(-1)^{k-1}}{k}\binom{n+1}{k}\{H_{n+1}-H_{n-k+1}\}\notag\\
&=\sum_{k=0}^{n-1}\frac{(-1)^k}{(k+1)\binom{n}{k+1}}+\sum_{k=0}^{n-1}\frac{(-1)^k}{(k+1)\binom{n}{k+1}}\{H_{k+1}-1\}\notag\\
&=\sum_{k=1}^{n}\frac{(-1)^{k-1}H_k}{k\binom{n}{k}}.
\end{align}
On the other hand, we have
\begin{align}\label{e:57} %=================================== Equation (3.25)
&\sum_{k=1}^{n}\frac{(-1)^{k-1}}{k}\binom{n+1}{k}\{H_{n+1}-H_{n-k+1}\}\notag\\
&=H_{n+1}\left[\sum_{k=1}^{n+1}\frac{(-1)^{k-1}}{k}\binom{n+1}{k}-\frac{(-1)^n}{n+1}\right]-\sum_{k=1}^{n+1}\frac{(-1)^{k-1}}{k}\binom{n+1}{k}H_{n+1-k}.
\end{align}
Using (\ref{e:43}) and (\ref{e:44}) we conclude from (\ref{e:57}) that
\begin{equation*}
\sum_{k=1}^{n}\frac{(-1)^{k-1}}{k}\binom{n+1}{k}\{H_{n+1}-H_{n-k+1}\}=\frac{(-1)^{n+1}H_{n+1}}{n+1}+\sum_{k=1}^{n+1}\frac{(-1)^{k-1}}{k^2\binom{n+1}{k}}.
\end{equation*}
Now the  proof follows from (\ref{e:56}).
\end{proof}
\textbf{Identity 16.} For $n\in\mathbb{N}$, we have
\begin{equation*}
\sum_{k=0}^{n}\frac{(-1)^{k-1}4^k\binom{n}{k}}{\binom{2k}{k}}=\frac{1}{2n-1}.
\end{equation*}
\begin{proof}
Setting $s=-\frac{1}{2}$ in (\ref{e:33}), we get
\begin{equation}\label{e:58} %=================================== Equation (3.26)
\sum_{k=0}^{n}(-1)^k\binom{n-1/2}{k}=(-1)^n\binom{n-3/2}{n}.
\end{equation}
\end{proof}
We have by (\ref{e:8}) and (\ref{e:9})
\begin{equation}\label{e:59} %=================================== Equation (3.27)
\binom{n-1/2}{k}=\frac{\Gamma(n+1/2)}{k!\Gamma(n+1/2)}=\frac{\Gamma(n-k+1/2)}{k!}\frac{(n-k)!2^{2n-2k}}{\Gamma(1/2)(2n-2k)!}
\end{equation}
and
\begin{equation}\label{e:60} %=================================== Equation (3.28)
\binom{n-3/2}{n}=\frac{\Gamma(n+1/2)}{(n-1/2)n!\Gamma(-1/2)}.
\end{equation}
Utilizing (\ref{e:59}) and (\ref{e:60}) in (\ref{e:58}), we may write
\begin{equation*}
\sum_{k=0}^{n}\frac{(-1)^{n-k}2^{2n-2k}n!(n-k)!}{k!(2n-2k)!}=\frac{2\Gamma(1/2)}{(2n-1)\Gamma(-1/2)}.
\end{equation*}
Noting that $\Gamma(-1/2)=-2\Gamma(1/2)$ and letting $k\to n-k$ in the left side, we complete the proof.
\begin{rem}
Identity 16 is known and can be found in \cite[Theorem 4.5]{22} .
\end{rem}
If we set $s=0$ in (\ref{e:25}) we get the following known result, see \cite{2}:\\
\textbf{Identity 17.} For $n\in\mathbb{N}$, we have
\begin{equation}\label{e:61} %=================================== Equation (3.29)
\sum_{k=1}^{n}\binom{n}{k}\frac{(-1)^{k-1}}{k^2}=\frac{H_n^2+H_n^{(2)}}{2}.
\end{equation}
\textbf{Identity 18.} For $n\in\mathbb{N}$, we have
\begin{equation*}
\sum_{k=1}^{n}\frac{(-1)^kH_{n-k}}{k\binom{n}{k}}=\frac{1-(-1)^n}{(n+1)^2}-\frac{H_n}{n+1}.
\end{equation*}
\begin{proof} If we take $s=1$ in (\ref{e:25}), we get
\begin{equation}\label{e:62} %=================================== Equation (3.30)
\sum_{k=1}^{n}\binom{n+1}{k}\frac{(-1)^{k-1}}{k^2}=\frac{H_n^2+H_n^{(2)}}{2}+\sum_{k=0}^{n-1}\frac{(-1)^k(H_n-H_k)}{(n-k)\binom{n}{k}}.
\end{equation}
Using (\ref{e:61}) we get
\begin{align}\label{e:63}%=================================== Equation (3.31)
\sum_{k=1}^{n}\binom{n+1}{k}\frac{(-1)^{k-1}}{k^2}&=\sum_{k=1}^{n+1}\binom{n+1}{k}\frac{(-1)^{k-1}}{k^2}-\frac{(-1)^n}{(n+1)^2}\notag\\
&=\frac{H_{n+1}^2+H_{n+1}^{(2)}}{2}-\frac{(-1)^n}{(n+1)^2}.
\end{align}
On the other hand we have
\begin{equation*}
\sum_{k=0}^{n-1}\frac{(-1)^k(H_n-H_k)}{(n-k)\binom{n}{k}}=H_n\sum_{k=0}^{n-1}\frac{(-1)^k}{(n-k)\binom{n}{k}}-\sum_{k=0}^{n-1}\frac{(-1)^kH_k}{(n-k)\binom{n}{k}}.
\end{equation*}
If we substitute $n-k=k'$ in  the sums on the right hand side and then dropping the prime on $k$, we get
\begin{equation*}
\sum_{k=0}^{n-1}\frac{(-1)^k(H_n-H_k)}{(n-k)\binom{n}{k}}=H_n\sum_{k=1}^{n}\frac{(-1)^{n-k}}{k\binom{n}{k}}-\sum_{k=1}^{n}\frac{(-1)^{n-k}H_{n-k}}{k\binom{n}{k}}.
\end{equation*}
Using (\ref{e:51}) gives us
\begin{equation}\label{e:64}%=================================== Equation (3.32)
\sum_{k=0}^{n-1}\frac{(-1)^k(H_n-H_k)}{(n-k)\binom{n}{k}}=\frac{H_n(1-(-1)^n)}{n+1}-(-1)^n\sum_{k=1}^{n}\frac{(-1)^kH_{n-k}}{k\binom{n}{k}}.
\end{equation}
Using (\ref{e:63}) and (\ref{e:64}) in (\ref{e:62}), we complete the proof.
\end{proof}
\textbf{Identity 19.} For all $n\in \mathbb{N}$, we have
\begin{equation}\label{e:65} %=================================== Equation (3.33)
\sum_{k=1}^{n}\frac{(-1)^{k-1}}{k^2}\binom{n}{k}H_{n-k}=\frac{H_n\left(H_n^2+H_n^{(2)}\right)}{2}-\sum_{k=0}^{n-1}\frac{(-1)^{k}(H_n-H_k)}{(k+1)(n-k)\binom{n}{k}}.
\end{equation}
\begin{proof} Setting $s=0$ in (\ref{e:31}), we get
\begin{equation}\label{e:66}%=================================== Equation (3.34)
\sum_{k=1}^{n}\frac{(-1)^{k-1}}{k^2}\binom{n}{k}\{H_n-H_{n-k}\}=\sum_{k=0}^{n-1}\frac{(-1)^k(H_n-H_k)}{(k+1)(n-k)\binom{n}{k}}.
\end{equation}
By (\ref{e:61}) we have
\begin{align}\label{e:67}%=================================== Equation (3.35)
\sum_{k=1}^{n}\frac{(-1)^{k-1}}{k^2}\binom{n}{k}\{H_n-H_{n-k}\}&=\frac{H_n\left(H_n)^2+H_n^{(2)}\right)}{2}\notag\\
 &-\sum_{k=1}^{n}\frac{(-1)^{k-1}}{k^2}\binom{n}{k}H_{n-k}.
 \end{align}
Combining Equations (\ref{e:66}) and (\ref{e:67}), we complete the proof of (\ref{e:65}).
\end{proof}
\textbf{Identity 20.} For all $n\in\mathbb{N}$, it holds that
\begin{align*}
&\sum_{k=1}^{n}\frac{(-1)^{k-1}}{k^2}\binom{n}{k}\left\{H_{n-k}^2+H_{n-k}^{(2)}\right\}=\frac{\left(H_n^2+H_n^{(2)}\right)^2}{2}\\
&-2\sum_{k=0}^{n-1}\frac{(-1)^k(H_n-H_k)^2}{(k+1)(n-k)\binom{n}{k}}.
\end{align*}
\begin{proof}
Setting $s=0$ in (\ref{e:32}), we get
\begin{align}\label{e:68}%=================================== Equation (3.36)
&\sum_{k=1}^{n}\frac{(-1)^{k-1}}{k^2}\binom{n}{k}\left\{(H_n-H_{n-k})^2+H_{n-k}^{(2)}-H_n^{(2)}\right\}\notag\\
&=2\sum_{k=0}^{n-1}\frac{(-1)^k(H_n-H_k)}{(k+1)(n-k)\binom{n}{k}}.
\end{align}
Clearly, we have
\begin{align*}
&\sum_{k=1}^{n}\frac{(-1)^{k-1}}{k^2}\binom{n}{k}\left\{(H_n-H_{n-k})^2+H_{n-k}^{(2)}-H_n^{(2)}\right\}\\
&=\left[H_n^2-H_n^{(2)}\right]\sum_{k=1}^{n}\frac{(-1)^{k-1}}{k^2}\binom{n}{k}-2H_n\sum_{k=1}^{n}\frac{(-1)^{k-1}}{k^2}\binom{n}{k}H_{n-k}\\
&+\sum_{k=1}^{n}\frac{(-1)^{k-1}}{k^2}\binom{n}{k}\left\{H_{n-k}^2+H_{n-k}^{(2)}\right\}.
\end{align*}
By the help of (\ref{e:61}) and (\ref{e:65}) we get from this identity
\begin{align*}
&\sum_{k=1}^{n}\frac{(-1)^{k-1}}{k^2}\binom{n}{k}\left\{(H_n-H_{n-k})^2+H_{n-k}^{(2)}-H_n^{(2)}\right\}\\
&=-\frac{\left(H_n^2+H_n^{(2)}\right)^2}{2}-2H_n\sum_{k=0}^{n-1}\frac{(-1)^{k-1}(H_n-H_k)}{(k+1)(n-k)\binom{n}{k}}\\
&+\sum_{k=1}^{n}\frac{(-1)^{k-1}}{k^2}\binom{n}{k}\left\{H_{n-k}^2+H_{n-k}^{(2)}\right\}.
\end{align*}
Now the proof follows from (\ref{e:68}).
\end{proof}

\end{document}